\documentclass[11pt]{amsart}

\usepackage[T1]{fontenc}
\usepackage{lmodern}
\usepackage{microtype}
\usepackage{amsmath,amssymb,amsthm,mathtools}
\usepackage{enumitem}
\usepackage[colorlinks=true,linkcolor=blue,citecolor=blue,urlcolor=blue]{hyperref}
\usepackage[nameinlink,noabbrev]{cleveref}

\allowdisplaybreaks
\numberwithin{equation}{section}

\newtheorem{theorem}{Theorem}[section]
\newtheorem{proposition}[theorem]{Proposition}
\newtheorem{lemma}[theorem]{Lemma}
\newtheorem{corollary}[theorem]{Corollary}
\newtheorem{remark}[theorem]{Remark}
\newtheorem{definition}[theorem]{Definition}

\newcommand{\R}{\mathbb{R}}
\newcommand{\C}{\mathbb{C}}
\newcommand{\T}{\mathbb{T}}
\newcommand{\dd}{\,\mathrm{d}}
\newcommand{\sgn}{\operatorname{sgn}}
\newcommand{\cstar}{c_{\star}}
\newcommand{\Aclass}{\mathcal{A}}

\title[An Improved Bound for the Ovals Problem]
{An Improved Bound for the Ovals Problem}

\author[D. Suragan]{Durvudkhan Suragan}
\address{
	Durvudkhan Suragan:
	\endgraf
	Department of Mathematics
	\endgraf
	Nazarbayev University
	\endgraf
	Astana 010000
	\endgraf
	Kazakhstan
	\endgraf
	{\it E-mail address:} {\rm  durvudkhan.suragan@nu.edu.kz}
}


\begin{document}

\begin{abstract}
Let $\gamma\subset\R^m,\,m\geq2,$ be a closed curve of length $2\pi$ with its curvature $\kappa$,
parametrized by arc length, and let $\lambda_\gamma$ be the first eigenvalue of
the periodic curvature Schr\"odinger operator
$-\dd^2/\dd s^2+\kappa(s)^2$. We obtain
\[
 \lambda_\gamma\geq \frac{\sqrt{\pi}}{2}
 \left(\frac{\Gamma(7/6)}{\Gamma(5/3)}\right)^3.
\]
This is a near-sharp lower bound for the Ovals problem.  
Our proof introduces a new geometric approach.  
We derive a convolution identity from the closure condition and combine it with projection averaging over tangent directions and sharp Poincar\'e  inequalities on antipodal arcs.
 As applications, we provide an improved two-state kinetic
Lieb--Thirring inequality and the corresponding two-eigenvalue constant.
\end{abstract}

\maketitle

\section{Introduction and main results}

Let $\gamma$ be a closed $W^{2,2}$ curve of length $2\pi$, parametrized by arc length, and let
$\kappa$ denote its curvature.  The associated periodic curvature Schr\"odinger
operator is
\begin{equation}\label{eq:oval-eigenvalue}
 H_\gamma:=-\frac{\dd^2}{\dd s^2}+\kappa(s)^2,
 \qquad
 \lambda_\gamma:=\inf_{0\neq\psi\in H^1(\T_{2\pi})}
 \frac{\displaystyle\int_0^{2\pi}
 \bigl(|\psi'|^2+\kappa^2|\psi|^2\bigr)\dd s}
 {\displaystyle\int_0^{2\pi}|\psi|^2\dd s},
\end{equation}
where  $ \lambda_\gamma$ is its first (least) eigenvalue and $\T_{2\pi}:=\R/2\pi\mathbb Z$.  For the class $\mathfrak C$ of smooth, simple,
closed, convex plane curves of length $2\pi$, set
\begin{equation}\label{eq:C-oval-def}
 C_{\mathrm{oval}}:=\inf_{\gamma\in\mathfrak C}\lambda_\gamma.
\end{equation}
The circle gives the upper bound $C_{\mathrm{oval}}\leq1$. The Ovals conjecture of Benguria and Loss \cite{BenguriaLoss2004} asserts
that $1$ is also a universal lower bound.  This is the quantity denoted by $C_{6.19}$ in Problem~6.19 (Ovals problem) of
\cite{GeorgievGomezSerranoTaoWagner2025}, where bounds as strong as possible
are requested.

Besides the circle, a family $\mathcal O$ of geometrically distinct convex ovals,
interpolating between the circle and a line segment traversed twice,
also has ground-state energy one
\cite{BenguriaLoss2004,BernsteinMettler2015}. The spectral quantity in
\eqref{eq:oval-eigenvalue} had already appeared in Burchard and Thomas's study
of dynamical Euler elastica, where 
$\lambda_\gamma\geq\frac14$ was proved \cite{BurchardThomas2003}.
Benguria and Loss subsequently formulated the Ovals conjecture $C_{\mathrm{oval}}=1$, exposed its
connection with the two-bound-state Lieb--Thirring problem, and established
$\lambda_\gamma\geq\frac12$ \cite{BenguriaLoss2004}.  Linde then improved
the universal lower bound to
\begin{equation}\label{eq:Linde-2006-bound}
 \lambda_\gamma>
 \left(1+\frac{\pi}{\pi+8}\right)^{-2}
 =0.6084773421\ldots
\end{equation}
in \cite{Linde2006}.  More recently, Linde established the analytic estimate
$\lambda_\gamma>0.81$ and a finite-dimensional variational bound whose numerical
evaluation gives $\lambda_\gamma>0.8246$ \cite{Linde2025}.   Burchard and Thomas
proved that the family $\mathcal O$ consists of local minimizers
\cite{BurchardThomas2005}, and Bernstein and Mettler characterized it through
projective differential geometry \cite{BernsteinMettler2015}. Denzler introduced a relaxed variational problem for which
compactness is restored, proved that minimizers exist, and showed that, with the
possible exception of the di-gon, every minimizer is represented by a planar,
real-analytic, strictly convex loop with strictly positive curvature
\cite{Denzler2015}.
The 2006 AIM workshop summary \cite{AIM2006}, the 2009 Oberwolfach report \cite{AshbaughBenguriaLaugesenWeidl2009}, and the survey
\cite{BenguriaLindeLoewe2012}
record further approaches to the Ovals problem.  The same problem was
recently used as a numerical benchmark in the work of Georgiev, G\'omez-Serrano, Tao, and Wagner
\cite[Section 8]{GeorgievGomezSerranoTaoWagner2025}; that computation supplied
no new rigorous lower bound. To the best of our knowledge, the sharp
lower bound in \eqref{eq:C-oval-def} therefore remains open. 

The main result of this paper is the following universal lower bound 
\begin{equation}\label{eq:cstar-def}
\lambda_\gamma \geq \cstar:=\frac{\sqrt{\pi}}{2}
 \left(\frac{\Gamma(7/6)}{\Gamma(5/3)}\right)^3
 =0.9618316523\ldots,
\end{equation}
where $\Gamma$ is the gamma function.
The approach developed here is different from previous analyses of the Ovals problem. Rather than studying the Euler--Lagrange equation or perturbations around minimizers, we derive a geometric convolution identity from the closure condition, combine it with an averaging argument over tangent directions, and reduce the problem to a H\"older estimate. This yields a universal lower bound for the curvature Schr\"odinger operator and, through a Bloch-vector construction, corresponding improvements for the two-state Lieb--Thirring inequality.

In fact, we prove a stronger, dimension-independent statement in Denzler's
harmonic coordinates \cite{Denzler2015}.  For $\Psi\in H^1(\T_{2\pi};\R^m)$, let
\[
 Z[\Psi]:=\{s\in\T_{2\pi}:\Psi(s)=0\},
 \qquad
 \sgn\Psi:=
 \begin{cases}
  \Psi/|\Psi|,&\Psi\neq0,\\
  0,&\Psi=0.
 \end{cases}
\]
\begin{theorem}\label{thm:relaxed-main}
Let $\Psi\in H^1(\T_{2\pi};\R^m)$, $m\geq2$.  Assume 
\begin{equation}\label{eq:weak-loop-intro}
 \left|\int_0^{2\pi}\sgn\Psi(s)\dd s\right|
 \leq |Z[\Psi]|,
\end{equation}
where $|Z[\Psi]|$ is the one-dimensional Lebesgue measure of the zero set.  Then
\begin{equation}\label{eq:relaxed-main}
 \int_0^{2\pi}|\Psi'(s)|^2\dd s
 \geq \cstar\int_0^{2\pi}|\Psi(s)|^2\dd s.
\end{equation}
\end{theorem}

The lower bound follows immediately and applies to a larger class than
that appearing in \eqref{eq:C-oval-def}.

\begin{corollary}\label{cor:curve-main}
Let $\gamma\subset\R^m,\,m\geq2,$ be a closed $W^{2,2}$ curve of length $2\pi$,
parametrized by arc length, and let $\kappa=|\gamma''|$ be its curvature.  Then
\begin{equation}\label{eq:curve-main}
 \lambda_\gamma\geq\cstar.
\end{equation}
In particular, $C_{\mathrm{oval}}\geq\cstar$.
\end{corollary}

We briefly describe the mechanism behind the estimate.  For a smooth strictly
convex plane curve, write its unit tangent as
$T(s)=(\cos\phi(s),\sin\phi(s))$ and introduce the vector ground-state transform
$Y=\psi T$.  For each direction $\alpha$, the scalar projection
$Y\cdot(-\sin\alpha,\cos\alpha)$ vanishes at the two points at which the tangent
directions are antipodal.  Applying the sharp Dirichlet Poincar\'e inequality on
the two complementary arcs and averaging in $\alpha$ produces an exact weighted
estimate involving their antipodal half-length $\ell(\alpha)$.  Geometric closure
implies the convolution identity
\begin{equation*}
 \int_0^\pi \ell(\beta-t)\sin t\dd t=2\pi,
\end{equation*}
and  H\"older's inequality then yields \eqref{eq:cstar-def}.
All the geometric and averaging identities entering this chain are exact;
the two inequalities are the Poincar\'e step and the H\"older step (see Remark \ref{rem:loss}).
Denzler's existence and regularity theory reduces the relaxed variational
problem to this convex planar case or to the di-gon, whose quotient is one.

The Ovals problem was introduced in connection with the one-dimensional
Lieb--Thirring conjecture.  Suppose $u_1,\ldots,u_N\in H^1(\R;\C)$ are orthonormal. Set
\begin{equation*}\label{eq:density-intro}
 \rho(x):=\sum_{j=1}^N|u_j(x)|^2.
\end{equation*}
Let $K_N$ be the optimal constant in
\begin{equation*}\label{eq:kinetic-LT-intro}
 \sum_{j=1}^N\int_\R|u_j'(x)|^2\dd x
 \geq K_N\int_\R\rho(x)^3\dd x.
\end{equation*}
Keller's sharp one-bound-state inequality, equivalently the sharp
one-dimensional Gagliardo--Nirenberg inequality, gives $K_1=\frac{\pi^2}{4}$;
see \cite{Keller1961,Weinstein1983}.  The Lieb--Thirring conjecture predicts $K_N=\pi^2/4$ for every $N$.  For background and the finite-rank
variational theory, see
\cite{LiebThirring1976,Frank2021,FrankLaptevWeidl2022,DolbeaultFelmerLossPaturel2006,FrankGontierLewin2021,FrankGontierLewin2025}.
Classical one-dimensional developments include the Aizenman--Lieb lifting
principle \cite{AizenmanLieb1978}, the kinetic argument of Eden and Foias \cite{EdenFoias1991}, Weidl's endpoint result \cite{Weidl1996}, and
the sharp endpoint theorem of Hundertmark, Lieb, and Thomas
\cite{HundertmarkLiebThomas1998}.
Benguria and Loss showed that the sharp Ovals inequality would imply the
conjectured two-state estimate for real-valued functions
\cite{BenguriaLoss2004}.  Their reduction carries an additional endpoint
condition, so the geometric problem is, a priori, stronger than the two-state
problem.

Theorem \ref{thm:relaxed-main} gives the following two-state kinetic Lieb--Thirring inequality.

\begin{theorem}\label{thm:two-state-main}
Let $u_1,u_2\in H^1(\R;\C)$ be orthonormal in $L^2(\R)$.  Then
\begin{equation}\label{eq:two-state-main}
 \int_\R\bigl(|u_1'|^2+|u_2'|^2\bigr)\dd x
 \geq \cstar\frac{\pi^2}{4}
 \int_\R\bigl(|u_1|^2+|u_2|^2\bigr)^3\dd x.
\end{equation}
\end{theorem}

Writing $u=(u_1,u_2)^{\mathsf T}$, the passage from a complex orthonormal pair to
a relaxed loop uses the Bloch vector
\begin{equation}\label{eq:bloch-intro}
 \mathbf b(x):=
 \begin{pmatrix}
  2\operatorname{Re}(\overline{u_1}u_2)\\
  2\operatorname{Im}(\overline{u_1}u_2)\\
  |u_1|^2-|u_2|^2
 \end{pmatrix}.
\end{equation}
It satisfies $|\mathbf b|=\rho$, $\int_\R\mathbf b=0$, and
\begin{equation}\label{eq:bloch-energy-intro}
 4\rho\bigl(|u_1'|^2+|u_2'|^2\bigr)-|\mathbf b'|^2
 =4\bigl(\operatorname{Im}\langle u,u'\rangle_{\C^2}\bigr)^2\geq0.
\end{equation}
After the cumulative-density change of variables, $\mathbf b$ becomes an
admissible $\R^3$-valued relaxed loop.  This both proves
Theorem \ref{thm:two-state-main} and explains why no reality assumption is needed.

By duality, we also obtain a bound for two negative eigenvalues.  For
$V\in L^{3/2}(\R)$ with $V\geq0$ almost everywhere, let
$\mu_1(V)\leq\mu_2(V)\leq\cdots$ denote the min--max values of
$-\dd^2/\dd x^2-V$ below the bottom of the essential spectrum, and set
\[
 \lambda_j(V):=\max\{-\mu_j(V),\,0\},
 \qquad j=1,2,
\]
so that $\lambda_1(V)\geq\lambda_2(V)\geq0$, with $\lambda_j(V)=0$ when the
$j$-th negative eigenvalue is absent.

\begin{corollary}\label{cor:two-eigenvalue}
For every $V\in L^{3/2}(\R)$, $V\geq0$,
\begin{equation}\label{eq:two-eigenvalue}
 \lambda_1(V)+\lambda_2(V)
 \leq \frac{4}{3\sqrt{3}\,\pi\sqrt{\cstar}}
 \int_\R V(x)^{3/2}\dd x.
\end{equation}
\end{corollary}

The constant in \eqref{eq:two-eigenvalue} equals $0.2498\ldots$.  For comparison,
the Eden--Foias argument \cite{EdenFoias1991} bounds the full eigenvalue
sum by $\tfrac{2}{3\sqrt3}\int_\R V^{3/2}=0.3849\ldots\int_\R V^{3/2}$, while Keller's
one-bound-state constant, which the Lieb--Thirring conjecture
predicts for the full sum, is $\tfrac{4}{3\sqrt3\,\pi}=0.2450\ldots$.  

The paper is organized as follows.  In Section \ref{sec:geometry} we derive the exact
closure identity for the antipodal half-length.  The projection argument and the
convex planar estimate are given in Section \ref{sec:projection}.  In
Section \ref{sec:relaxed} we combine this estimate with Denzler's argument to prove Theorem \ref{thm:relaxed-main} and Corollary \ref{cor:curve-main}.  The
Bloch-vector reduction and Theorem \ref{thm:two-state-main} with Corollary \ref{cor:two-eigenvalue} are established in
Section \ref{sec:bloch}.  The
weighted cumulative-density lemma used in the reduction is recorded in
Appendix \ref{sec:mass-appendix}.

\section{Convex loops and the technical lemma}\label{sec:geometry}

Throughout this section, $\gamma$ is a smooth, strictly convex, closed plane curve of length $2\pi$, parametrized counterclockwise by arc length, and its curvature
is assumed to be everywhere positive:
\[
 \gamma:\T_{2\pi}\longrightarrow\R^2,
 \qquad |\gamma'(s)|=1.
\]
Write
\begin{equation}\label{eq:tangent-angle}
 T(s):=\gamma'(s)=\bigl(\cos\phi(s),\sin\phi(s)\bigr).
\end{equation}
The positive-curvature assumption permits a smooth lift of the tangent angle
satisfying
\begin{equation}\label{eq:phi-properties}
 \phi'(s)=\kappa(s)>0,
 \qquad
 \phi(s+2\pi)=\phi(s)+2\pi\nu,
\end{equation}
where $\nu\geq1$ is the turning number of $\gamma$.  Since a closed convex plane
curve is by definition the boundary of a convex body, it is simple and
$\nu=1$; this is the normalization used throughout, and it is the only case
needed, because the curves supplied by Proposition \ref{prop:Denzler-reduction}
are convex in this sense.  Thus \eqref{eq:phi-properties} reads
$\phi(s+2\pi)=\phi(s)+2\pi$.
Let $s=\sigma(\theta)$ denote the inverse map, normalized so that
\begin{equation}\label{eq:sigma-periodic}
 \sigma(\theta+2\pi)=\sigma(\theta)+2\pi.
\end{equation}
Set
\begin{equation}\label{eq:g-def}
 g(\theta):=\sigma'(\theta)
 =\frac1{\kappa(\sigma(\theta))}.
\end{equation}
Then $g$ is positive and $2\pi$-periodic, and
\begin{equation}\label{eq:g-mean}
 \int_0^{2\pi}g(\theta)\dd\theta=2\pi.
\end{equation}
The closure condition $\int_0^{2\pi}T(s)\dd s=0$ becomes
\begin{equation}\label{eq:g-closure}
 \int_0^{2\pi}e^{i\theta}g(\theta)\dd\theta=0.
\end{equation}
\begin{definition}\label{def:half-length}
For $\alpha\in\R$, define
\begin{equation}\label{eq:ell-def}
 \ell(\alpha)
 :=\sigma(\alpha+\pi)-\sigma(\alpha)
 =\int_\alpha^{\alpha+\pi}g(\theta)\dd\theta.
\end{equation}
Thus $\ell(\alpha)$ is the length of the oriented arc whose endpoint tangent directions are $e^{i\alpha}$ and $-e^{i\alpha}$.
\end{definition}

The complementary arc has length $\ell(\alpha+\pi)$, and hence
\begin{equation}\label{eq:ell-antipodal}
 \ell(\alpha)+\ell(\alpha+\pi)=2\pi.
\end{equation}
The next identity is the main geometric input.
\begin{lemma}\label{lem:closure-convolution}
Let $g\in L^1(\T_{2\pi})$ satisfy $g\geq0$, \eqref{eq:g-mean} and
\eqref{eq:g-closure}, and let $\ell$ be given by \eqref{eq:ell-def}.  Then, for
every $\beta\in\R$,
\begin{equation}\label{eq:closure-convolution-complex}
 \int_0^\pi\ell(\beta-t)\,e^{it}\dd t=2\pi i;
\end{equation}
equivalently,
\begin{equation}\label{eq:closure-convolution}
 \int_0^\pi\ell(\beta-t)\sin t\dd t=2\pi
 \quad\text{and}\quad
 \int_0^\pi\ell(\beta-t)\cos t\dd t=0.
\end{equation}
\end{lemma}

\begin{proof}
By \eqref{eq:ell-def},
$\ell(\beta-t)=\int_0^{2\pi}g(\theta)\,\mathbf 1_{E(\beta-t)}(\theta)\dd\theta$,
where $E(\alpha)\subset\T_{2\pi}$ is the arc from $\alpha$ to $\alpha+\pi$.
Since $g\geq0$ is integrable, Fubini's theorem gives
\begin{equation}\label{eq:closure-fubini}
 \int_0^\pi\ell(\beta-t)e^{it}\dd t
 =\int_0^{2\pi}g(\theta)\,w(\beta-\theta)\dd\theta,
 \quad
 w(u):=\int_{(0,\pi)\cap(u,u+\pi)}e^{it}\dd t,
\end{equation}
here the inner intersection being taken modulo $2\pi$. Indeed
$\theta\in E(\beta-t)$ is equivalent to $t\in(u,u+\pi)$ modulo $2\pi$, with
$u:=\beta-\theta$.

We claim that
\begin{equation}\label{eq:w-formula}
 w(u)=i\bigl(1+e^{iu}\bigr)
 \qquad\text{for every }u\in\R.
\end{equation}
By periodicity it suffices to check $u\in[0,2\pi)$.  If $u\in[0,\pi]$, then
$(0,\pi)\cap(u,u+\pi)=(u,\pi)$ and
\[
 w(u)=\int_u^\pi e^{it}\dd t
 =\frac{e^{i\pi}-e^{iu}}{i}
 =i\bigl(1+e^{iu}\bigr).
\]
If $u\in(\pi,2\pi)$, then $(u,u+\pi)$ meets $(0,\pi)$ in $(0,u-\pi)$ and
\[
 w(u)=\int_0^{u-\pi}e^{it}\dd t
 =\frac{e^{i(u-\pi)}-1}{i}
 =i\bigl(1+e^{iu}\bigr),
\]
which proves \eqref{eq:w-formula}.

Substituting \eqref{eq:w-formula} into \eqref{eq:closure-fubini} and using that
$g$ is real-valued,
\begin{align*}
 \int_0^\pi\ell(\beta-t)e^{it}\dd t
 &=i\int_0^{2\pi}g(\theta)\dd\theta
 +ie^{i\beta}\int_0^{2\pi}g(\theta)e^{-i\theta}\dd\theta\\
 &=2\pi i
 +ie^{i\beta}\,\overline{\int_0^{2\pi}g(\theta)e^{i\theta}\dd\theta},
\end{align*}
by \eqref{eq:g-mean}.  The remaining integral vanishes by the closure condition
\eqref{eq:g-closure}, which proves
\eqref{eq:closure-convolution-complex}; taking imaginary and real parts gives
\eqref{eq:closure-convolution}.
\end{proof}

\section{Projection averaging and the convex planar bound}\label{sec:projection}
For $\psi\in H^1(\T_{2\pi})$, define the vector-valued ground-state transform
\begin{equation}\label{eq:Y-def}
 Y(s):=\psi(s)T(s).
\end{equation}
Since $T$ is real-valued with $T\cdot T'=0$ and $|T'|=\kappa$, one has, for every
$\psi\in H^1(\T_{2\pi};\C)$, the pointwise identity
\begin{equation}\label{eq:Y-energy}
 |Y'|^2=|\psi'|^2+\kappa^2|\psi|^2.
\end{equation}

For $\alpha\in\R$, let
\begin{equation}\label{eq:nu-alpha}
 \nu_\alpha:=(-\sin\alpha,\cos\alpha).
\end{equation}
Recall that $\sigma$ is the inverse of the lifted tangent-angle map.  Put
\begin{equation}\label{eq:s-alpha}
 s_\alpha:=\sigma(\alpha)
 \quad\text{and}\quad
 I_\alpha:=\bigl\{s\in\T_{2\pi}:
 s_\alpha<s<s_{\alpha+\pi}\bigr\},
\end{equation}
where the inequality is understood on the positively oriented lift to $\R$.
Thus $I_\alpha$ is the open oriented arc starting at the unique point whose
tangent is $e^{i\alpha}$ and ending at the unique point whose tangent is
$-e^{i\alpha}$.  Its length is
\[
 |I_\alpha|=s_{\alpha+\pi}-s_\alpha=\ell(\alpha).
\]
The complementary open arc is $I_{\alpha+\pi}$ and has length
$\ell(\alpha+\pi)=2\pi-\ell(\alpha)$.

Consider the scalar projection
\begin{equation}\label{eq:f-alpha}
 f_\alpha(s):=Y(s)\cdot\nu_\alpha
 =\psi(s)\sin(\phi(s)-\alpha).
\end{equation}
Because $\phi(s_\alpha)=\alpha$ and
$\phi(s_{\alpha+\pi})=\alpha+\pi$, the trace of $f_\alpha$ vanishes at both
endpoints of $I_\alpha$ and also at both endpoints of the complementary arc
$I_{\alpha+\pi}$.

\begin{proposition}\label{prop:averaged-two-arc}
For every $\psi\in H^1(\T_{2\pi})$,
\begin{equation}\label{eq:averaged-two-arc}
 \int_0^{2\pi}\bigl(|\psi'|^2+\kappa^2|\psi|^2\bigr)\dd s
 \geq2\pi\int_0^{2\pi}J(\phi(s))|\psi(s)|^2\dd s,
\end{equation}
where
\begin{equation}\label{eq:J-def}
 J(\beta):=\int_0^\pi
 \frac{\sin^2t}{\ell(\beta-t)^2}\dd t.
\end{equation}
\end{proposition}

\begin{proof}
We first assume that $\psi$ is real-valued.  For a fixed $\alpha$, the
restrictions of $f_\alpha$ to $I_\alpha$ and $I_{\alpha+\pi}$ belong to the
corresponding Dirichlet Sobolev spaces because their endpoint traces vanish.
The sharp one-dimensional Dirichlet Poincar\'e inequality on an interval of
length $L$ is
\begin{equation}\label{eq:Dirichlet-Poincare}
 \int_0^L|h'(r)|^2\dd r
 \geq\frac{\pi^2}{L^2}\int_0^L|h(r)|^2\dd r,
 \qquad h\in H_0^1(0,L).
\end{equation}
Applying \eqref{eq:Dirichlet-Poincare} separately to the two arcs and adding the
results gives
\begin{align}
 \int_0^{2\pi}|f_\alpha'(s)|^2\dd s
 \geq \pi^2\bigg(&\frac1{\ell(\alpha)^2}
 \int_{I_\alpha}f_\alpha(s)^2\dd s\notag\\
 &+\frac1{\ell(\alpha+\pi)^2}
 \int_{I_{\alpha+\pi}}f_\alpha(s)^2\dd s\bigg).
 \label{eq:two-arc-poincare}
\end{align}
The endpoints have measure zero, so the two arc integrals on the left add to
the integral over the whole circle.

We now integrate \eqref{eq:two-arc-poincare} with respect to
$\alpha\in[0,2\pi]$.  Since $\nu_\alpha$ is independent of $s$,
$f_\alpha'(s)=Y'(s)\cdot\nu_\alpha$. Also, for every $v=(v_1,v_2)\in\R^2$, we have 
\begin{equation*}\label{eq:angular-average-vector}
 \int_0^{2\pi}|v\cdot\nu_\alpha|^2\dd\alpha
 =\int_0^{2\pi}|-v_1\sin\alpha+v_2\cos\alpha|^2\dd\alpha
 =\pi|v|^2.
\end{equation*}
Thus, the integral of the left-hand side of \eqref{eq:two-arc-poincare} gives
\begin{equation}\label{eq:average-left}
 \int_0^{2\pi}\int_0^{2\pi}|f_\alpha'(s)|^2\dd s\dd\alpha
 =\pi\int_0^{2\pi}|Y'(s)|^2\dd s.
\end{equation}

It remains to compute the integral of the right-hand side of \eqref{eq:two-arc-poincare}.  Consider first the contribution
of $I_\alpha$:
\begin{equation}\label{eq:R1-def}
 R_1:=\pi^2\int_0^{2\pi}\frac1{\ell(\alpha)^2}
 \int_{I_\alpha}f_\alpha(s)^2\dd s\dd\alpha.
\end{equation}
All integrands are nonnegative and integrable, so Fubini's theorem permits us to reverse the
order of integration.  Fix $s\in\T_{2\pi}$ and choose the lift
$\beta:=\phi(s)\in\R$.  By the strict monotonicity of $\phi$, the membership
condition $s\in I_\alpha$ is equivalent, modulo $2\pi$, to
\begin{equation}\label{eq:arc-membership}
 \alpha<\beta<\alpha+\pi,
 \qquad\text{or equivalently}\qquad
 \beta-\pi<\alpha<\beta.
\end{equation}
Consequently, using $f_\alpha(s)^2=\psi(s)^2\sin^2(\beta-\alpha)$ and then the
substitution $t=\beta-\alpha$, we obtain
\begin{align}
 R_1
 &=\pi^2\int_0^{2\pi}\psi(s)^2
 \left(
 \int_{\beta-\pi}^{\beta}
 \frac{\sin^2(\beta-\alpha)}{\ell(\alpha)^2}\dd\alpha
 \right)\dd s\notag\\
 &=\pi^2\int_0^{2\pi}\psi(s)^2
 \left(
 \int_0^\pi\frac{\sin^2t}{\ell(\beta-t)^2}\dd t
 \right)\dd s\notag\\
 &=\pi^2\int_0^{2\pi}J(\phi(s))\psi(s)^2\dd s.
\notag
\end{align}

The contribution of the complementary arc is
\[
 R_2:=\pi^2\int_0^{2\pi}\frac1{\ell(\alpha+\pi)^2}
 \int_{I_{\alpha+\pi}}f_\alpha(s)^2\dd s\dd\alpha.
\]
Set $\alpha'=\alpha+\pi$.  Since $I_{\alpha+\pi}=I_{\alpha'}$,
$\ell(\alpha+\pi)=\ell(\alpha')$, and
\[
 f_{\alpha'-\pi}(s)^2
 =\psi(s)^2\sin^2(\phi(s)-\alpha'+\pi)
 =\psi(s)^2\sin^2(\phi(s)-\alpha')
 =f_{\alpha'}(s)^2,
\]
periodicity in $\alpha'$ shows that $R_2=R_1$.  Hence the right side of
\eqref{eq:two-arc-poincare} equals
\[
 2\pi^2\int_0^{2\pi}J(\phi(s))\psi(s)^2\dd s.
\]
Combining this identity with \eqref{eq:average-left}, dividing by $\pi$, and
using \eqref{eq:Y-energy} gives \eqref{eq:averaged-two-arc} for real $\psi$.
For complex $\psi$, we apply the real-valued result to $\operatorname{Re}\psi$ and
$\operatorname{Im}\psi$ and add the two inequalities.
\end{proof}


\begin{lemma}\label{lem:J-lower}
For every $\beta\in\R$,
\begin{equation}\label{eq:J-lower}
 J(\beta)\geq\frac{\cstar}{2\pi}.
\end{equation}
\end{lemma}

\begin{proof}
Let
\begin{equation}\label{eq:A-def}
 A:=\int_0^\pi\sin^{4/3}t\dd t.
\end{equation}
With
\[
 \sin^{4/3}t
 =\left(\frac{\sin^2t}{\ell(\beta-t)^2}\right)^{1/3}
 \left(\ell(\beta-t)\sin t\right)^{2/3}
\]
H\"older's inequality (with exponents $3$ and $3/2$) gives
\begin{align}
 A
 &\leq
 \left(\int_0^\pi\frac{\sin^2t}{\ell(\beta-t)^2}\dd t\right)^{1/3}
 \left(\int_0^\pi\ell(\beta-t)\sin t\dd t\right)^{2/3}\notag\\
 &=J(\beta)^{1/3}(2\pi)^{2/3},
 \label{eq:Holder-J}
\end{align}
where Lemma \ref{lem:closure-convolution} was used in the last step. Therefore
\[
 J(\beta)\geq\frac{A^3}{(2\pi)^2}=\frac{\cstar}{2\pi}.
\]

The beta-function identity
\begin{equation*}
 A=\sqrt{\pi}\,\frac{\Gamma(7/6)}{\Gamma(5/3)}
\end{equation*}
turns $\cstar$ into the expression in \eqref{eq:cstar-def}.

\end{proof}


\begin{theorem}\label{thm:convex-oval}
Let $\gamma$ be a smooth, strictly convex, closed plane curve of length $2\pi$ with everywhere positive curvature. Then, for every $\psi\in H^1(\T_{2\pi})$,
\begin{equation}\label{eq:convex-oval}
 \int_0^{2\pi}\bigl(|\psi'|^2+\kappa^2|\psi|^2\bigr)\dd s
 \geq\cstar\int_0^{2\pi}|\psi|^2\dd s.
\end{equation}
In particular, $\lambda_\gamma\geq\cstar$.
\end{theorem}

\begin{proof}
Insert \eqref{eq:J-lower} into \eqref{eq:averaged-two-arc}.
\end{proof}

\begin{remark}\label{rem:equality-family}
If $\phi(s+\pi)=\phi(s)+\pi$, then $\ell\equiv\pi$, hence
$J\equiv(2\pi)^{-1}$, and Proposition~\ref{prop:averaged-two-arc} returns the
sharp value $\lambda_\gamma\geq1$.  This applies to the circle. It also applies to every member of the family $\mathcal O$: for such a curve the ground-state
transform has the form $Y(s)=Ae(s)$, where $e(s)=(\cos s,\sin s)$ and $A$ is an
invertible $2\times2$ matrix \cite{BenguriaLoss2004,BernsteinMettler2015}.  Thus $Y(s+\pi)=-Y(s)$; since
$|Y(s+\pi)|=|Y(s)|$, one obtains $T(s+\pi)=-T(s)$ and hence
$\phi(s+\pi)=\phi(s)+\pi$.  In the measure-valued limiting description of the
di-gon, $g=\pi\delta_0+\pi\delta_\pi$ likewise gives $\ell\equiv\pi$.  Thus the
projection-averaging stage recovers the value $1$ on the entire known equality
family. That is,  the argument is
therefore exact on the entire known equality set of the Ovals problem.
\end{remark}

 
\begin{remark}\label{rem:loss}
The chain of the proof is
\[
 \lambda_\gamma
 \;\geq\;2\pi\min_\beta J(\beta)
 \;\geq\;\cstar ,
\]
the first inequality being Proposition \ref{prop:averaged-two-arc} and the second
Lemma \ref{lem:J-lower}.  Indeed it is a similar proof structure we followed for eigenvalue lower bounds in our recent paper \cite{FrankLaptevS2026}.

The angular averaging is exact, since
$\int_0^{2\pi}|Y'(s)\cdot\nu_\alpha|^2\mathrm d\alpha=\pi |Y'(s)|^2,$
so no component of $Y'$ is lost at that stage. The losses in our argument arise from the two-arc Poincar\'e inequalities and the subsequent H\"older estimate.
 These have consequences for how far the present method can be pushed.  The sharp lower bound for the Ovals problem remains open.
\end{remark}

\section{Proofs of Theorem \ref{thm:relaxed-main} and Corollary \ref{cor:curve-main}}\label{sec:relaxed}

We now pass from convex planar curves to arbitrary relaxed loop. Define
\begin{equation}\label{eq:Aclass}
 \Aclass_m:=\left\{
 \begin{array}{l|l}
 \Psi\in H^1(\T_{2\pi};\R^m)\setminus\{0\}
 &\displaystyle
 \left|\int_0^{2\pi}\sgn\Psi\dd s\right|
 \leq |Z[\Psi]|,
 \end{array}
 \right\},  \; m\geq2.
\end{equation}
This is exactly the constraint appearing in the relaxed oval problem in harmonic
coordinates \cite[Formula (2.6)]{Denzler2015}, where it is called the weak loop
condition. Denzler introduced it in order to restore compactness for the Ovals
variational problem. We use the following consequence of his existence,
classification, planarity, and regularity theorems.

\begin{proposition}\label{prop:Denzler-reduction}
The variational quantity
\begin{equation}\label{eq:relaxed-inf}
 \mu_m:=\inf_{\Psi\in\Aclass_m}
 \frac{\displaystyle\int_0^{2\pi}|\Psi'|^2\dd s}
 {\displaystyle\int_0^{2\pi}|\Psi|^2\dd s}
\end{equation}
is attained, and $\mu_m\leq1$. Any minimizer is of one of the following two types:
\begin{enumerate}[label=\textup{(\roman*)}]
\item it is a di-gon configuration, that is, $Z[\Psi]$ is either a closed interval
of length exactly $\pi$ or a pair of points at distance $\pi$, and the associated
curve is a straight segment of length $\pi$ traversed back and forth; in this case
the Rayleigh quotient in \eqref{eq:relaxed-inf} equals $1$;
\item it is zero-free and, after an orthogonal change of coordinates in $\R^m$, has the form
\begin{equation}\label{eq:Denzler-form}
 \Psi(s)=\varphi(s)T(s),
 \qquad \varphi(s)>0,
\end{equation}
where $T$ is the unit tangent field of a planar, real-analytic, strictly convex closed curve of length $2\pi$ with strictly positive curvature.
\end{enumerate}
\end{proposition}

\begin{proof}
The variational problem \eqref{eq:relaxed-inf} is Denzler's relaxed oval problem
in harmonic coordinates \cite[formula (2.6)]{Denzler2015}: the functional is
$\int_0^{2\pi}|\Psi'|^2$ under the $L^2$ normalization and the weak loop
condition, the normalization being immaterial by homogeneity.  By
\cite[Theorem 3.1]{Denzler2015} the infimum is attained by some $\Psi$ with
Rayleigh quotient at most $1$, and the zero set $Z[\Psi]$ falls into one of four
cases: $Z[\Psi]=\varnothing$; $Z[\Psi]$ a closed interval of length $<\pi$
(``$D$-shaped''); $Z[\Psi]$ a closed interval of length exactly $\pi$; or
$Z[\Psi]$ a pair of points at distance $\pi$.  The last two are the di-gon and
give case (i).  The $D$-shaped case does not occur, by
\cite[Theorem 3.4]{Denzler2015}.  In the remaining case $Z[\Psi]=\varnothing$,
\cite[Theorem 3.3]{Denzler2015} states that $\Psi$ represents a planar, strictly
convex, real-analytic curve with strictly positive curvature; writing
$\varphi:=|\Psi|>0$ and $T:=\Psi/|\Psi|$ gives \eqref{eq:Denzler-form}, and
$\gamma(s):=\int_0^sT$ is closed of length $2\pi$ because the weak loop condition
reads $\int_0^{2\pi}T\dd s=0$ when $|Z[\Psi]|=0$.  This is case (ii).  See also
\cite[Theorem 1.1]{Denzler2015} for the combined statement.

Finally, in case (i) the Rayleigh quotient equals $1$: writing $\Psi=\varphi T$
with $T$ constant equal to $\pm e_1$ on each of the two arcs of length $\pi$, one
has $\int_0^{2\pi}|\Psi'|^2=\int_0^{2\pi}|\varphi'|^2$ with $\varphi\in H_0^1$ on
each arc, so the sharp Poincar\'e inequality \eqref{eq:Dirichlet-Poincare} with
$L=\pi$ gives $\int|\Psi'|^2\geq\int|\Psi|^2$, with equality for
$\varphi=|\sin s|$.
\end{proof}

\begin{proof}[Proof of Theorem \ref{thm:relaxed-main}]
By homogeneity, it is enough to prove that $\mu_m\geq\cstar$. Let $\Psi$ be a minimizer from Proposition \ref{prop:Denzler-reduction}. In the di-gon case, its quotient is $1>\cstar$. Otherwise, write $\Psi=\varphi T$ as in \eqref{eq:Denzler-form}. Since $T$ is a unit tangent, $|T'|=\kappa$ and $TT'=0$,
\begin{equation}\label{eq:Denzler-energy-decomp}
 |\Psi'|^2=|\varphi'|^2+\kappa^2\varphi^2.
\end{equation}
The associated curve is smooth, planar, and strictly convex, so Theorem \ref{thm:convex-oval} yields
\[
 \int_0^{2\pi}|\Psi'|^2\dd s
 \geq\cstar\int_0^{2\pi}|\Psi|^2\dd s.
\]
Thus $\mu_m\geq\cstar$, and the same inequality holds for every element of $\Aclass_m$.
\end{proof}

\begin{proof}[Proof of Corollary \ref{cor:curve-main}]
Let $\psi$ be a ground state for the operator in \eqref{eq:oval-eigenvalue}; it may be chosen strictly positive. Define
\[
 \Psi(s):=\psi(s)\gamma'(s).
\]
Then $Z[\Psi]=\varnothing$ and
\[
 \int_0^{2\pi}\sgn\Psi(s)\dd s
 =\int_0^{2\pi}\gamma'(s)\dd s=0,
\]
so $\Psi\in\Aclass_m$. Moreover, $|\gamma'|=1$ implies
$\gamma'\cdot\gamma''=0$ almost everywhere, and 
\[
 |\Psi'|^2=|\psi'|^2+\kappa^2\psi^2.
\]
Applying Theorem \ref{thm:relaxed-main} gives \eqref{eq:curve-main}. This proves the claim for $W^{2,2}$ curves.
\end{proof}

\section{Proofs of Theorem \ref{thm:two-state-main} and Corollary \ref{cor:two-eigenvalue}}\label{sec:bloch}

Let
\begin{equation}\label{eq:u-vector}
 u(x):=\begin{pmatrix}u_1(x)\\u_2(x)\end{pmatrix}\in\C^2,
 \quad
 \rho(x):=|u(x)|^2=|u_1(x)|^2+|u_2(x)|^2.
\end{equation}
We use the convention that the Hermitian inner product is linear in its second argument. Introduce the Pauli matrices
\begin{equation}\label{eq:Pauli}
 \sigma_1=\begin{pmatrix}0&1\\1&0\end{pmatrix},
 \qquad
 \sigma_2=\begin{pmatrix}0&-i\\i&0\end{pmatrix},
 \qquad
 \sigma_3=\begin{pmatrix}1&0\\0&-1\end{pmatrix},
\end{equation}
and define the Bloch vector by
\begin{equation}\label{eq:bloch-def}
 b_k(x):=\langle u(x),\sigma_k u(x)\rangle_{\C^2},
 \qquad k=1,2,3.
\end{equation}
This agrees with \eqref{eq:bloch-intro}.

\begin{lemma}\label{lem:bloch-identities}
For almost every $x\in\R$,
\begin{align}
 |\mathbf b(x)|^2&=\rho(x)^2,
 \label{eq:bloch-length}\\
 |\mathbf b'(x)|^2
 &=4\rho(x)|u'(x)|^2
 -4\bigl(\operatorname{Im}\langle u(x),u'(x)\rangle_{\C^2}\bigr)^2
 \leq4\rho(x)|u'(x)|^2.
 \label{eq:bloch-differential}
\end{align}
If $u_1,u_2$ are orthonormal, then
\begin{equation}\label{eq:bloch-mean-zero}
 \int_{\R}\mathbf b(x)\dd x=0.
\end{equation}
\end{lemma}

\begin{proof}
The Pauli matrices satisfy the completeness relation
\begin{equation}\label{eq:Pauli-completeness}
 \sum_{k=1}^3(\sigma_k)_{ij}(\sigma_k)_{\ell r}
 =2\delta_{ir}\delta_{j\ell}-\delta_{ij}\delta_{\ell r}.
\end{equation}
Applying \eqref{eq:Pauli-completeness} to $u$ gives \eqref{eq:bloch-length}. Differentiating $b_k$ yields
\[
 b_k'=2\operatorname{Re}\langle u,\sigma_k u'\rangle.
\]
A second use of \eqref{eq:Pauli-completeness} gives
\[
 \sum_{k=1}^3|\langle u,\sigma_k u'\rangle|^2
 =2\rho|u'|^2-|\langle u,u'\rangle|^2.
\]
Moreover,
\[
 \sum_{k=1}^3\langle u,\sigma_k u'\rangle^2
 =\langle u,u'\rangle^2.
\]
Consequently,
\begin{align*}
 |\mathbf b'|^2
 &=2\sum_{k=1}^3|\langle u,\sigma_k u'\rangle|^2
 +2\operatorname{Re}\sum_{k=1}^3\langle u,\sigma_k u'\rangle^2\\
 &=4\rho|u'|^2-4\bigl(\operatorname{Im}\langle u,u'\rangle\bigr)^2,
\end{align*}
which is \eqref{eq:bloch-differential}. Thus, we arrive at
\begin{align*}
 \int_{\R}b_1\dd x&=2\operatorname{Re}\langle u_1,u_2\rangle=0,\\
 \int_{\R}b_2\dd x&=2\operatorname{Im}\langle u_1,u_2\rangle=0,\\
 \int_{\R}b_3\dd x&=\|u_1\|_2^2-\|u_2\|_2^2=0,
\end{align*}
proving \eqref{eq:bloch-mean-zero}.
\end{proof}

The products defining $\rho$ and $\mathbf b$ belong to $W^{1,1}(\R)$.
Moreover, \eqref{eq:bloch-differential} implies
$\int_{\R}|\mathbf b'|^2/\rho<\infty$, with the integrand set equal to zero on
$\{\rho=0\}$.  Thus the hypotheses of Lemma~\ref{lem:mass-coordinate} are
satisfied. 
Since $u_1,u_2$ are orthonormal, we have 
$$
 \int_{\R}\rho(x)\dd x=2.
$$
Define the cumulative-density coordinate
$$
 t=t(x):=\frac12\int_{-\infty}^x\rho(y)\dd y,
 \qquad 0<t<1,
$$
and let $x=x(t)$ denote its generalized inverse. Put
\begin{equation}\label{eq:B-def}
 \mathbf B(t):=\mathbf b(x(t)).
\end{equation}
 Lemma \ref{lem:mass-coordinate} gives
\begin{align}
 \int_0^1|\mathbf B'(t)|^2\dd t
 &=2\int_{\R}\frac{|\mathbf b'(x)|^2}{\rho(x)}\dd x
 \leq8\int_{\R}|u'(x)|^2\dd x,
 \label{eq:B-energy}\\
 \int_0^1|\mathbf B(t)|^2\dd t
 &=\frac12\int_{\R}\rho(x)^3\dd x.
 \label{eq:B-L2}
\end{align}
Moreover, by \eqref{eq:bloch-length},
\begin{equation}\label{eq:B-sign}
 \sgn\mathbf B(t)=\frac{\mathbf b(x(t))}{\rho(x(t))}
\end{equation}
for almost every $t$, and hence
\begin{equation}\label{eq:B-loop}
 \int_0^1\sgn\mathbf B(t)\dd t
 =\frac12\int_{\R}\mathbf b(x)\dd x=0.
\end{equation}
The endpoint traces of $\mathbf B$ vanish, since $|\mathbf b|=\rho$ and
one-dimensional $H^1$ functions tend to zero at infinity; hence the periodic
extension of $\mathbf B$ is continuous and $\mathbf B\in H^1(\T_1;\R^3)$.  By
\eqref{eq:B-loop} it satisfies the weak loop condition \eqref{eq:weak-loop-intro}
in the strengthened form $\int_0^1\sgn\mathbf B\dd t=0$, irrespective of the size
of $Z[\mathbf B]$.

\begin{proof}[Proof of Theorem \ref{thm:two-state-main}]
Rescale $\mathbf B$ to a loop of length $2\pi$ by setting
\begin{equation}\label{eq:Psi-rescale}
 \Psi(s):=\mathbf B\!\left(\frac{s}{2\pi}\right).
\end{equation}
By \eqref{eq:B-loop}, $\Psi$ is admissible in Theorem \ref{thm:relaxed-main}. Therefore
\begin{equation}\label{eq:B-relaxed-scaled}
 \int_0^1|\mathbf B'|^2\dd t
 \geq4\pi^2\cstar\int_0^1|\mathbf B|^2\dd t.
\end{equation}
Combining \eqref{eq:B-energy}, \eqref{eq:B-L2}, and \eqref{eq:B-relaxed-scaled} gives
\[
 8\int_{\R}|u'|^2\dd x
 \geq4\pi^2\cstar\cdot\frac12\int_{\R}\rho^3\dd x.
\]
Since $|u'|^2=|u_1'|^2+|u_2'|^2$, this gives \eqref{eq:two-state-main}.
\end{proof}

Now we are in the position to prove the dual two-eigenvalue bound.
Let
$$
 a:=\cstar\frac{\pi^2}{4}.
$$
For $v,r\geq0$, a direct calculation gives
\begin{equation}\label{eq:young-cubic}
 vr-ar^3
 \leq\frac{2}{3\sqrt{3a}}v^{3/2}
 =\frac{4}{3\sqrt{3}\,\pi\sqrt{\cstar}}v^{3/2}.
\end{equation}
Indeed, the maximum is attained at $r=(v/(3a))^{1/2}$.

\begin{proof}[Proof of Corollary \ref{cor:two-eigenvalue}]
The (postitive-part) Ky Fan variational principle gives
\begin{equation}\label{eq:Ky-Fan-two}
 \lambda_1(V)+\lambda_2(V)
 =\sup_{\substack{u_1,u_2\in H^1(\R;\C)\\
 \langle u_i,u_j\rangle=\delta_{ij}}}
 \left\{\int_{\R}V\rho\dd x
 -\sum_{j=1}^2\int_{\R}|u_j'|^2\dd x\right\}.
\end{equation}

By Theorem~\ref{thm:two-state-main}, the expression in braces is bounded above by
\[
 \int_{\R}\bigl(V(x)\rho(x)-a\rho(x)^3\bigr)\dd x.
\]
Applying \eqref{eq:young-cubic} pointwise and integrating proves
\eqref{eq:two-eigenvalue}.
\end{proof}

\appendix

\section{Weighted cumulative-density coordinate}\label{sec:mass-appendix}

For completeness, we record the generalized-inverse change of variables used in Section \ref{sec:bloch}. The following statement 
 is formulated in a form adapted to the Bloch vector.

\begin{lemma}\label{lem:mass-coordinate}
Let $M>0$, let $\rho\in W^{1,1}_{\mathrm{loc}}(\R)$ satisfy
\begin{equation}\label{eq:mass-lemma-rho}
 \rho\geq0,
 \qquad
 \int_{\R}\rho\dd x=M,
 \qquad
 \rho(x)\longrightarrow0\quad\text{as }|x|\to\infty,
\end{equation}
and let $b\in W^{1,1}_{\mathrm{loc}}(\R;\R^m)$ satisfy
\begin{equation}\label{eq:mass-lemma-b}
 |b(x)|=\rho(x)\quad\text{a.e.},
 \qquad
 \int_{\R}\frac{|b'(x)|^2}{\rho(x)}\dd x<\infty,
\end{equation}
where the integrand is set equal to zero on $\{\rho=0\}$. Define
\begin{equation}\label{eq:mass-lemma-t}
 t(x):=\frac1M\int_{-\infty}^x\rho(y)\dd y
\end{equation}
and let $x(t)$ be a generalized inverse. Then $B(t):=b(x(t))$ has a representative in $H_0^1(0,1;\R^m)$ and
\begin{align}
 \int_0^1|B'(t)|^2\dd t
 &=M\int_{\R}\frac{|b'(x)|^2}{\rho(x)}\dd x,
 \label{eq:mass-chain-energy}\\
 \int_0^1|B(t)|^2\dd t
 &=\frac1M\int_{\R}|b(x)|^2\rho(x)\dd x,
 \label{eq:mass-chain-L2}\\
 \int_0^1\sgn B(t)\dd t
 &=\frac1M\int_{\R}\sgn b(x)\rho(x)\dd x.
 \label{eq:mass-chain-sign}
\end{align}
\end{lemma}

\begin{proof}
Since $|b|=\rho$, the chain rule for the Euclidean norm gives
\begin{equation}\label{eq:rho-prime-b-prime}
 |\rho'(x)|\leq |b'(x)|
 \qquad\text{for almost every }x.
\end{equation}

For $\varepsilon>0$, apply the Sobolev chain rule to
$F_\varepsilon(\rho):=(\rho+\varepsilon)^{1/2}-\varepsilon^{1/2}$. Since
$0\leq F_\varepsilon(\rho)\leq\sqrt\rho$ and $F_\varepsilon(\rho)\to\sqrt\rho$ in $L^2(\R)$, weak lower semicontinuity yields
\begin{equation}\label{eq:sqrt-rho-H1}
 \int_{\R}|(\sqrt\rho\,)'|^2\dd x
 \leq\frac14\int_{\R}\frac{|b'|^2}{\rho}\dd x<\infty.
\end{equation}
Together with $\int\rho=M$, this shows that $\sqrt\rho\in H^1(\R)$. In particular, $\rho$ is bounded and $\rho^3\in L^1(\R)$.

The Cauchy--Schwarz inequality gives
\[
 \int_\R|b'|\dd x
 \leq M^{1/2}\left(\int_\R\frac{|b'|^2}{\rho}\dd x\right)^{1/2},
\]
so $b\in W^{1,1}(\R;\R^m)$ and, by \eqref{eq:rho-prime-b-prime},
$\rho\in W^{1,1}(\R)$.

Let $\Omega:=\{x:\rho(x)>0\}$. Since $\rho$ has a continuous, locally absolutely continuous representative, write
\[
 \Omega=\bigcup_\nu I_\nu
\]
as a countable disjoint union of open intervals. The map $x\mapsto t(x)$ is absolutely continuous and strictly increasing on every $I_\nu$; denote its image by $J_\nu$ and its inverse by $x_\nu:J_\nu\to I_\nu$. The intervals $J_\nu$ are disjoint and cover $(0,1)$ up to a null set, because
\begin{equation}\label{eq:Jnu-length}
 |J_\nu|=\frac1M\int_{I_\nu}\rho(x)\dd x,
 \qquad
 \sum_\nu |J_\nu|=1.
\end{equation}
For almost every $t\in J_\nu$,
\begin{equation}\label{eq:inverse-derivative}
 x_\nu'(t)=\frac{M}{\rho(x_\nu(t))}.
\end{equation}
Define $B(t):=b(x_\nu(t))$ on $J_\nu$ and $B=0$ on the null complement of $\bigcup_\nu J_\nu$. The standard chain rule gives
\begin{equation}\label{eq:Bprime-chain}
 B'(t)=\frac{M b'(x_\nu(t))}{\rho(x_\nu(t))}
 \qquad\text{a.e.\ on }J_\nu.
\end{equation}
Therefore
\begin{align*}
 \sum_\nu\int_{J_\nu}|B'(t)|^2\dd t
 &=M\int_{\Omega}\frac{|b'(x)|^2}{\rho(x)}\dd x,\\
 \sum_\nu\int_{J_\nu}|B(t)|^2\dd t
 &=\frac1M\int_{\Omega}|b(x)|^2\rho(x)\dd x.
\end{align*}
At every finite endpoint of a component $I_\nu$, continuity and maximality of the component give $b=0$; at infinite endpoints this follows from $|b|=\rho\to0$. Hence $B|_{J_\nu}$ has zero endpoint traces. Its zero extension belongs to $H_0^1(J_\nu;\R^m)$, and summability of the displayed energies shows that the componentwise zero extensions glue to a function in $H_0^1(0,1;\R^m)$. This proves \eqref{eq:mass-chain-energy} and \eqref{eq:mass-chain-L2}. Hence applying the same componentwise change of variables to the bounded function $\sgn b$ gives \eqref{eq:mass-chain-sign}.
\end{proof}

\subsection*{Acknowledgements}
The author acknowledges the use of AI tools. All mathematical arguments and proofs in the final manuscript were checked and written by the author.

\subsection*{Funding}
This research is funded by Nazarbayev University under the FDCRGP 110326FD3205 (D.S.).



\end{document}